\documentclass[a4paper,11pt]{article}

\title{\Large On values of zeta functions of Arakawa-Kaneko type}
\author{Tomoko Hoshi}
\date{}
\usepackage[top=30truemm,bottom=30truemm,left=30truemm,right=30truemm]{geometry}
\usepackage{amsmath,amssymb,amsfonts}
\usepackage{amsthm}
\usepackage{color}
\usepackage{latexsym}
\def\qed{\hfill $\Box$} 
\newtheorem{defi}{Definition}[section]
\newtheorem{thm}{Theorem}[section]
\newtheorem{lemma}{Lemma}[section]

\newtheorem{ex}{Example}[section]
\newtheorem{rem}{Remark}[section]

\makeatletter
\def\underbrace@#1#2{\vtop {\m@th \ialign {##\crcr $\hfil #1{#2}\hfil $\crcr \noalign {\kern 3\p@ \nointerlineskip }\upbracefill \crcr \noalign {\kern 3\p@ }}}}
\def\overbrace@#1#2{\vbox {\m@th \ialign {##\crcr \noalign {\kern 3\p@ }\downbracefill \crcr \noalign {\kern 3\p@ \nointerlineskip }$\hfil #1 {#2}\hfil $\crcr }}}

\def\underbrace#1{%
   \mathop{\mathchoice{\underbrace@{\displaystyle}{#1}}
   {\underbrace@{\textstyle}{#1}}
   {\underbrace@{\scriptstyle}{#1}}
   {\underbrace@{\scriptscriptstyle}{#1}}}\limits
}
\def\overbrace#1{%
   \mathop{\mathchoice{\overbrace@{\displaystyle}{#1}}
   {\overbrace@{\textstyle}{#1}}
   {\overbrace@{\scriptstyle}{#1}}
   {\overbrace@{\scriptscriptstyle}{#1}}}\limits
}
\makeatother

\allowdisplaybreaks

\newcommand{\stiri}[2]{\genfrac{[}{]}{0pt}{0}{#1}{#2}}
\newcommand{\stirii}[2]{\genfrac{\{}{\}}{0pt}{0}{#1}{#2}}

\begin{document}

\maketitle

\begin{abstract}
For these two decades, the Arakawa-Kaneko zeta function has been studied actively. Recently Kaneko and Tsumura constructed its variants from the viewpoint of poly-Bernoulli numbers. In this paper, we generalize their zeta functions of Arakawa-Kaneko type to those with indices in which positive and negative integers are mixed. We show that values of these functions at positive integers can be expressed in terms of the multiple Hurwitz zeta star values.
\end{abstract}


\section{Introduction}

The Arakawa-Kaneko zeta function is defined by
\begin{equation}
\label{xieq}
\xi(k_1, k_2,\ldots,k_r;s) = \frac{1}{\Gamma(s)} \int_{0}^{\infty} t^{s-1} \frac{{\rm Li}_{k_1,k_2, \ldots,k_r}(1-e^{-t})}{e^t-1}dt \ \ \ \ \ \ (\Re(s)>0)
\end{equation}
for $k_1,k_2,\ldots,k_r \in \mathbb{Z}_{\geq 1}$ in \cite{[AK]}.
This zeta function has been studied and generalized by a lot of authors (see, for example, \cite{[BH],[BH2],[CC],[CC2],[H],[H2],[KST],[KO],[KT2],[Y],[Y2]}).

It should be noted that even if we define $\xi(-k_1, -k_2,\ldots ,-k_r;s)$ by replacing $\{ k_{j} \}$ by $\{ -k_j \}$ in (\ref{xieq}), this does not converge for any $s\in \mathbb{C}$.

As variants of \eqref{xieq}, Kaneko and Tsumura \cite{[KT]} defined the following $\eta$ and $\widetilde{\xi}$ functions by
\begin{equation}
\label{eta+eq}
\eta(k_1,k_2, \ldots,k_r;s) = \frac{1}{\Gamma(s)} \int_{0}^{\infty} t^{s-1} \frac{{\rm Li}_{k_1,k_2, \ldots,k_r}(1-e^t)}{1-e^t}dt \ \ \ \ \ \ (\Re(s)>0)
\end{equation}
for $k_1,k_2,\ldots,k_r \in \mathbb{Z}_{\geq 1}$,
\begin{equation}
\label{eta-eq}
\eta(-k_1,-k_2, \ldots,-k_r;s) = \frac{1}{\Gamma(s)} \int_{0}^{\infty} t^{s-1} \frac{{\rm Li}_{-k_1,-k_2, \ldots,-k_r}(1-e^t)}{1-e^t}dt \ \ \ \ \ \ (\Re(s)>0)
\end{equation}
for $k_1,k_2,\ldots,k_r \in \mathbb{Z}_{\geq 0}$, and 
\begin{equation}
\label{tildexieq}
\widetilde{\xi}(-k_1,-k_2,\ldots,-k_r;s) = \frac{1}{\Gamma(s)} \int_{0}^{\infty} t^{s-1} \frac{{\rm Li}_{-k_1,-k_2,\ldots,-k_r}(1-e^{t})}{e^{-t}-1}dt \ \ \ \ \ (\Re(s)>0)
\end{equation}
for $k_1,k_2,\ldots,k_r \in \mathbb{Z}_{\geq0}$ with $(k_1,k_2,\ldots,k_r) \neq (0,\ldots,0)$.

It is also noted that we cannot define $\widetilde{\xi}(k_1,k_2,\ldots,k_r;s)$ by replacing $\{ -k_j \}$ by $\{ k_j \}$ in (\ref{tildexieq}) for $(k_j) \in \mathbb{Z}^{r}_{\geq1}$.

We emphasize that indices of these zeta functions consist of all positive integers or all nonpositive integers.

In this paper, we aim to consider the cases of these zeta functions with indices in which positive and negative integers are mixed. 
To explain our results in some detail, we give two overviews of necessary background.
First, we recall the multiple zeta values (MZVs, for short) given by
\begin{equation}
\label{mzv}
\zeta(p_1, p_2, \ldots, p_n) = \sum_{0<m_1<m_2< \cdots <m_n} \frac{1}{m_1^{p_1}m_2^{p_2}\cdots m_n^{p_n}},
\end{equation}
and the multiple zeta star values (MZSVs, for short) are given by
\begin{equation}
\label{mzsv}
\zeta^{\star}(p_1, p_2, \ldots, p_n) = \sum_{0< m_1\leq m_2\leq \cdots \leq m_n} \frac{1}{m_1^{p_1}m_2^{p_2}\cdots m_n^{p_n}}
\end{equation}
for $p_1,p_2,\ldots,p_n \in \mathbb{Z}_{\geq 1}$ with $p_n \geq 2$.
Further, the multiple Hurwitz zeta star values are given by
\begin{eqnarray}
\label{mhzsv}
&& \zeta^{\star}(p_1,p_2,\ldots,p_n;\alpha_1,\alpha_2,\ldots,\alpha_n) \nonumber \\
&&= \sum_{0 \leq m_1 \leq m_2 \leq \cdots \leq m_n}^{} \frac{1}{(m_1+\alpha_1)^{p_1} (m_2+\alpha_2)^{p_2} \cdots (m_n+\alpha_n)^{p_n}}
\end{eqnarray}
for $p_1,p_2,\ldots,p_n \in \mathbb{Z}_{\geq 1}$ with $p_n \geq 2$ and $\alpha_1,\alpha_2, \ldots, \alpha_n>0$.
Let $\{ \alpha \}^n= (\underbrace{\alpha, \ldots, \alpha}_{n})$.
In particular, we have $\zeta^{\star}(p_1, \ldots ,p_n; \{ 1 \}^n)=\zeta^{\star}(p_1, \ldots ,p_n)$.

Secondly, we recall multi-poly-Bernoulli numbers. Imatomi, Kaneko and Takeda defined 
\begin{align}
\frac{{\rm Li}_{k_1,k_2, \ldots, k_r}(1-e^{-t})}{1-e^{-t}} = \sum_{n=0}^{\infty} B_n^{(k_1,k_2, \ldots,k_r)}\frac{t^n}{n!} \\
\frac{{\rm Li}_{k_1,k_2, \ldots, k_r}(1-e^{-t})}{e^t-1} = \sum_{n=0}^{\infty} C_n^{(k_1,k_2, \ldots,k_r)}\frac{t^n}{n!}
\end{align}
for $k_1, k_2, \ldots, k_r \in \mathbb{Z}$ in \cite{[IKT]}, where
\begin{align}
{\rm Li}_{k_1, k_2, \ldots, k_r}(z) = \sum_{0<m_1<m_2< \cdots < m_r} \frac{z^{m_r}}{m^{k_1}_1 m^{k_2}_2 \cdots m^{k_r}_r} \ \ \ \ \ \ (|z|<1)
\end{align}
is the multiple polylogarithm.
In particular, we have $B_n^{(1)} = (-1)^n C_n^{(1)}$ and these are ``ordinary'' Bernoulli numbers.
Kaneko and Tsumura defined another type of multi-poly-Bernoulli numbers by
\begin{eqnarray}
&& \sum_{a=0}^{r-1} (-1)^a \binom{r-1}{a} \sum_{l_1, \cdots, l_r \geq 1}^{} \frac{\prod_{j=1}^{r}(1-e^{-\sum_{\nu=j}^{r}x_\nu})^{l_j-1}}{(l_1+\cdots+l_r-a)^s} \nonumber \\
&& = \sum_{m_1,\ldots,m_r \geq0}^{}\mathfrak{B}_{m_1,\ldots,m_r}^{(s)} \frac{{x_1}^{m_1} \cdots {x_r}^{m_r}}{m_1! \cdots m_r!}
\end{eqnarray}
for $s \in \mathbb{C}$ in \cite{[KT]}.
The case of $r=1$ gives $\mathfrak{B}^{(k)}_{m}=B^{(k)}_{m}$ for $k \in \mathbb{Z}$.

We see that (\ref{xieq}), (\ref{eta+eq}), (\ref{eta-eq}) and (\ref{tildexieq}) can be analytically continued as entire functions. And the values of these functions at nonpositive integers are given by
\begin{eqnarray}
\xi(k_1, k_2, \ldots ,k_r ;-m) &=& (-1)^mC_m^{(k_1,k_2, \ldots, k_r)}, \\
\eta(k_1, k_2, \ldots ,k_r ;-m) &=&  B_m^{(k_1,k_2, \ldots, k_r)}, \\
\eta(-k_1, -k_2, \ldots, -k_r;-m) &=& B_m^{(-k_1,-k_2, \ldots, -k_r)}, \\
\widetilde{\xi}(-k_1,-k_2,\ldots,-k_r;-m) &=& C_m^{(-k_1,-k_2,\ldots,-k_r)}
\end{eqnarray}
for $m \in \mathbb{Z}_{\geq 0}$.
Arakawa-Kaneko \cite{[AK]} and Kaneko-Tsumura \cite{[KT]} constructed relations for the values of $\xi$ and $\eta$ with MZVs and MZSVs. 
For example, the following relations are important.
For $m \in \mathbb{Z}_{\geq0}$, $r$, $k \in \mathbb{Z}_{\geq 1}$,
\begin{equation}
\label{xim+1eq}
\xi(\underbrace{1,\ldots,1}_{r-1},k;m+1) = \sum_{\substack{a_1+\cdots+a_k=m \\ \forall a_j\geq0}} \binom{a_k+r}{r} \zeta(a_1+1,\ldots,a_{k-1}+1,a_k+r+1),
\end{equation}
\begin{equation}
\label{etam+1eq}
\eta(\underbrace{1,\ldots,1}_{r-1},k;m+1) = (-1)^{r-1} \sum_{\substack{a_1+\cdots+a_k=m \\ \forall a_j\geq0}} \binom{a_k+r}{r} \zeta^{\star}(a_1+1,\ldots,a_{k-1}+1,a_k+r+1).
\end{equation}
Further, Kaneko and Tsumura obtained the following result which is a kind of the duality formula.
For $k_1,\ldots,k_r \in \mathbb{Z}_{\geq 0}$, we have
\begin{equation}
\label{etaseq}
\eta(-k_1, -k_2, \ldots, -k_r;s) = \mathfrak{B}_{k_1, k_2, \ldots, k_r} ^{(s)} \ \ \ \ \ (s\in \mathbb{C}).
\end{equation}
Therefore, for $m \in \mathbb{Z}_{\geq 0}$, we obtain
\begin{equation}
\label{BB}
B_m^{(-k_1,-k_2,\ldots,-k_r)} = \mathfrak{B}^{(-m)}_{k_1,k_2, \ldots, k_r}
\end{equation}
(see \cite[Theorem 4.7]{[KT]}).

Using these results, we define the special types of $\eta$, ${\xi}$ and $\widetilde{\xi}$ functions (see \S \ref{sec-2}-\ref{sec-4}). 
We explicitly give the relations for the values of these functions with MZVs, MZSVs and the multiple Hurwitz zeta star values (see Theorems \ref{eta(k,-n;m+1)}, \ref{eta(-n,k;m+1)}, \ref{eta(1,...,1,-n;m+1)}, \ref{xi(-n,k;m+1)}, \ref{tildexi(k,-n;m+1)}).


\section{$\eta$ function}\label{sec-2}

In this section, we consider $\eta$ function, whose indices consist of positive integers and nonpositive integers.

\subsection{$\eta(k, -n;s)$}

\begin{defi} \label{eta(k,-n;s)}
For $k \in \mathbb{Z}_{\geq 1}$ and $n \in \mathbb{Z}_{\geq 0}$, define
\begin{equation}
\label{eta(k,-n;s)eq}
\eta(k,-n;s) = \frac{1}{\Gamma(s)} \int_{0}^{\infty} t^{s-1} \frac{{\rm Li}_{k,-n}(1-e^t)}{1-e^t}dt
\end{equation}
for $s \in \mathbb{C}$ with $\Re(s)>0$.
\end{defi}

The integral on the right-hand side converges absolutely in the domain $\Re(s)>0$, as is seen from the following lemma.

\begin{lemma} \label{Li_{k,-n}}
For $k \in \mathbb{Z}_{\geq1}$ and $n \in \mathbb{Z}_{\geq0}$ with $k \geq n$, we have
\begin{eqnarray}
\label{Li_{k,-n}eq1}
{\rm Li}_{k,-n}(z)
&=& z \left\{ \frac{P^{(n)}_{n-1}(z)}{(1-z)^{n+1}}{\rm Li}_{k}(z) + \frac{P^{(n)}_{n-2}(z)}{(1-z)^{n}}{\rm Li}_{k-1}(z) + \cdots \right. \nonumber \\
&& \left. + \cdots + \frac{P^{(n)}_{0}(z)}{(1-z)^{2}}{\rm Li}_{k-n+1}(z) + \frac{1}{1-z}{\rm Li}_{k-n}(z) \right\},
\end{eqnarray}
where $P^{(n)}_{i}(z)$ is the polynomial defined by
\begin{equation}
\label{P_i}
P^{(n)}_{i}(z) = \binom{n}{i+1} \sum_{j=0}^{i} \sum_{l=0}^{j+1} (-1)^{l} \binom{i+2}{l} (j-l+1)^{i+1} z^{i-j}
\end{equation}
for $0\leq i \leq n-1$.
\end{lemma}

\noindent
{\it Proof of Lemma \ref{Li_{k,-n}}.}
We can verify
\[
{\rm Li}_{k,-n}(z) = \sum_{j=0}^{n} \binom{n}{j} {\rm Li}_{-n+j}(z) {\rm Li}_{k-j}(z)
\]
by induction.
Further, we know the formula
\[
{\rm Li}_{-n+j}(z) = \frac{\mathcal{E}_{n-j}(z)}{(1-z)^{n-j+1}} \ \ \ \ (j \leq n),
\]
where ${\mathcal{E}}_{i}(z)$ is the Eulerian polynomial given by
\[
\mathcal{E}_{i}(z) = \sum_{j=0}^{i-1} \sum_{l=0}^{j+1} (-1)^{l} \binom{i+1}{l} (j-l+1)^{i} z^{i-j}
\]
(see \cite[(1.2), (1.3)]{[BH]}, \cite{[C]}).
Setting $P_i^{(n)}(z)=\binom{n}{i+1}\mathcal{E}_{i+1}(z)\frac{1}{z}$, we can prove this lemma.
\qed

\begin{ex} \label{exeta(k,-n;s)}
By Definition \ref{eta(k,-n;s)} and Lemma \ref{Li_{k,-n}}, we can derive $\eta(1,0;s)=-s$, and
\begin{align*}
&\eta(2,0;1)=-\zeta(2), \\
&\eta(2,0;2)=-\zeta(2)-2\zeta(3), \\
&\eta(2,-1;1)=-\frac{1}{2}-\frac{1}{2}\zeta(2), \\
&\eta(3,0;1)=-2\zeta(3),  \\
&\eta(3,-1;1) = -\frac{1}{2}\zeta(2) -\zeta(3).
\end{align*}
\end{ex}

\begin{thm} \label{eta(k,-n;-m)}
The function $\eta(k,-n;s)$ can be analytically continued as an entire function, and the values  at nonpositive integers are given by 
\begin{equation}
\eta(k,-n;-m) = B_m^{(k,-n)} \ \ \ \ \ \ (m \in \mathbb{Z}_{\geq 0}).
\end{equation}
\end{thm}

\begin{proof}
Let $C$ be the standard contour, namely the path consisting of the positive real axis from $\infty$ to $\varepsilon$, a counterclockwise circle $C_{\varepsilon}$ around the origin of radius $\varepsilon$, and the positive real axis from $\varepsilon$ to $\infty$.
Let
\begin{eqnarray*}
H(k,-n;s)
&=& \int_{C} t^{s-1} \frac{{\rm Li}_{k,-n}(1-e^t)}{1-e^t}dt \nonumber \\
&=& (e^{2\pi is}-1) \int_\varepsilon^\infty t^{s-1} \frac{{\rm Li}_{k,-n}(1-e^t)}{1-e^t}dt + \int_{C_\varepsilon}  t^{s-1} \frac{{\rm Li}_{k,-n}(1-e^t)}{1-e^t}dt.
\end{eqnarray*}
It follows from Lemma \ref{Li_{k,-n}} that $H(k,-n;s)$ is entire, because the integrand has no singularity on $C$ and the contour integral is absolutely convergent for all $s \in \mathbb{C}$. When we suppose $\Re(s)>0$, we can see
\begin{eqnarray*}
\int_{C_\varepsilon} t^{s-1} \frac{{\rm Li}_{k,-n}(1-e^t)}{1-e^t} dt \xrightarrow[\varepsilon \to0]{} 0.
\end{eqnarray*}
Hence
\begin{eqnarray*}
\eta(k,-n;s)
&=& \frac{1}{(e^{2 \pi is}-1) \Gamma(s)} H(k,-n;s),
\end{eqnarray*}
which can be analytically continued to $\mathbb{C}$, and is entire. In fact $\eta(k,-n;s)$ is holomorphic for $\Re(s)>0$, hence has no singularity at any positive integer.
Note that
\begin{eqnarray*}
\frac{1}{(e^{2 \pi is}-1)\Gamma(s)} \xrightarrow[s\to-m]{} \frac{(-1)^mm!}{2 \pi i} \ \ \ \ (m \in \mathbb{Z}_{\geq 0}).
\end{eqnarray*}
Setting $s=-m \in \mathbb{Z}_{\leq 0}$, we have $\eta(k,-n;-m) = B^{(k,-n)}_m$.
This completes the proof.
\end{proof}

Concerning the values at positive integer arguments, we prove the following theorem.

\begin{thm} \label{eta(k,-n;m+1)}
For $m \in \mathbb{Z}_{\geq 0}, k \in \mathbb{Z}_{\geq 1}$ and $n \in \mathbb{Z}_{\geq 0}$ with $k>n$, we have
\begin{eqnarray}
\label{eta(k,-n;m+1)eq}
&& \eta(k,-n;m+1) \nonumber \\ 
&=& -\sum_{l=0}^{n-1} \sum_{j=0}^{n-l-1} \sum_{\substack{a_{l+1}+\cdots+a_{k}=m \\ \forall a_i\geq0}} A_{l,j}^{(n)}(a_k+1) \frac{1}{(n-l-j+1)^{a_{l+1}+1}} \nonumber \\
&& \times \zeta^{\star} (a_{l+2}+1, \ldots , a_{k-1}+1, a_{k}+2;\{ n-l-j+1 \}^{k-l-1}) \nonumber \\
&& -\sum_{\substack{a_{n+1}+\cdots+a_{k}=m \\ \forall a_i\geq0}} (a_k+1) \zeta^{\star} (a_{n+2}+1, \ldots ,a_{k-1}+1, a_{k}+2;\{ 1 \}^{k-n-1}),
\end{eqnarray}
where $A_{l,j}^{(n)}$ is the rational number defined by
\begin{equation}
A_{l,j}^{(n)} = \binom{n}{l} \sum_{b=0}^{n-l-j-1} \sum_{d=0}^{b+1} (-1)^{d+j} \binom{n-l+1}{d} \binom{n-l-b-1}{j} (b-d+1)^{n-l}.
\end{equation}
\end{thm}

In order to prove the theorem, we give the following integral expression, which can be similarly proved as \cite[Lemma 2.7]{[KT]}.

\begin{lemma} \label{calculation1}
\label{lemma}
For $a_{l+1},a_{l+2}, \ldots ,a_{k-1} ,a_{k} \in \mathbb{Z}_{\geq 1}$ and $n \in \mathbb{N}$, we have
\begin{eqnarray*}
&& \frac{1}{\prod_{m=l+1}^{k}\Gamma(a_m+1)} \int_{0}^{\infty} \cdots \int_{0}^{\infty} x_{l+1}^{a_{l+1}} \cdots x_{k-1}^{a_{k-1}}x_k^{a_k+1} \\
&& \times \frac{e^{x_k}}{e^{x_k}-1} \frac{e^{x_{k-1}+x_k}}{e^{x_{k-1}+x_k}-1} \cdots \frac{e^{x_{l+2}+\cdots+{x_k}}}{e^{x_{l+2}+\cdots+{x_k}}-1} \frac{1}{e^{n(x_{l+1}+ \cdots +x_{k})}} dx_{l+1} \cdots dx_k \\
&=& \frac{a_k+1}{n^{a_{l+1}+1}} \zeta^{\star}(a_{l+2}+1, a_{l+3}+1, \ldots ,a_{k-1}+1, a_k+2 ; \{n\}^{k-l-1}).
\end{eqnarray*}
\end{lemma}

\noindent
{\it Proof of Theorem \ref{eta(k,-n;m+1)}.}
By setting $P^{(n)}_{i}(z)=\sum_{j=0}^{i}A^{(n)}_{n-i-1,j}(1-z)^j$ in Lemma $\ref{Li_{k,-n}}$, we obtain
\begin{eqnarray*}
A^{(n)}_{i,j} =  \binom{n}{i} \sum_{b=0}^{n-i-j-1} \sum_{d=0}^{b+1}  (-1)^{d+j} \binom{n-i+1}{d} \binom{n-i-b-1}{j} (b-d+1)^{n-i}.
\end{eqnarray*}
Consider the integral
\begin{eqnarray*}
I(s;k,-n)
&=& \frac{1}{\Gamma(s)} \sum_{l=0}^{n-1} \sum_{j=0}^{n-l-1} A^{(n)}_{l,j} \int_{0}^{\infty} \int_{0}^{u_k} \cdots \int_{0}^{u_{l+2}} u_k^{s-1} u_{l+1} \frac{e^{u_{l+1}}}{e^{u_{l+1}}-1} \cdots \frac{e^{u_{k-1}}}{e^{u_{k-1}}-1} \\
&& \times \frac{1}{e^{(n-l-j+1)u_k}} du_{l+1} \cdots du_{k} \\
&& + \frac{1}{\Gamma(s)} \int_{0}^{\infty} \int_{0}^{v_k} \cdots \int_{0}^{v_{n+2}} v_k^{s-1} v_{n+1} \frac{e^{v_{n+1}}}{e^{v_{n+1}}-1} \cdots \frac{e^{v_{k-1}}}{e^{v_{k-1}}-1} \\
&& \times \frac{1}{e^{v_k}} dv_{n+1}  \cdots dv_k.
\end{eqnarray*}
We can transform this formula as follows.
\begin{eqnarray*}
I(s;k,-n)
&=& \frac{-1}{\Gamma(s)} \sum_{l=0}^{n-1} \sum_{j=0}^{n-l-1} A^{(n)}_{l,j} \int_{0}^{\infty} u_k^{s-1} \int_{0}^{u_k} \cdots \left\{ \int_{0}^{u_{l+2}} \frac{{\rm Li}_{1}(1-e^{u_{l+1}})e^{u_{l+1}}}{e^{u_{l+1}}-1} du_{l+1} \right\} \\
&& \times \frac{e^{u_{l+2}}}{e^{u_{l+2}}-1} \cdots \frac{e^{u_{k-1}}}{e^{u_{k-1}}-1} \frac{1}{e^{(n-l-j+1)u_k}} du_{l+2} \cdots du_{k} \\
&& + \frac{-1}{\Gamma(s)} \int_{0}^{\infty} v_k^{s-1} \int_{0}^{v_k} \cdots \left\{ \int_{0}^{v_{n+2}} \frac{{\rm Li}_{1}(1-e^{v_{n+1}})e^{v_{n+1}}}{e^{v_{n+1}}-1} dv_{n+1} \right\} \\
&& \times \frac{e^{v_{n+2}}}{e^{v_{n+2}}-1} \cdots \frac{e^{v_{k-1}}}{e^{v_{k-1}}-1} \frac{1}{e^{v_k}} dv_{n+2} \cdots dv_k \\
&=& \frac{-1}{\Gamma(s)} \sum_{l=0}^{n-1} \sum_{j=0}^{n-l-1} A^{(n)}_{l,j} \int_{0}^{\infty} u_k^{s-1} \frac{{\rm Li}_{k-l}(1-e^{u_k})}{e^{(n-l-j+1)u_k}} du_k \\
&& + \frac{-1}{\Gamma(s)} \int_{0}^{\infty} v_k^{s-1} \frac{{\rm Li}_{k-n}(1-e^{v_k})}{e^{v_k}} dv_k \\
&=& \frac{-1}{\Gamma(s)} \sum_{l=0}^{n-1} \int_{0}^{\infty} u_k^{s-1} \frac{P^{(n)}_{n-l-1}(1-e^{u_k})}{e^{(n-l+1)u_k}} {\rm Li}_{k-l}(1-e^{u_k}) du_k \\
&& + \frac{-1}{\Gamma(s)} \int_{0}^{\infty} v_k^{s-1} \frac{1}{e^{v_k}} {\rm Li}_{k-n}(1-e^{v_k}) dv_k \\
&=& -\eta(k,-n;s).
\end{eqnarray*}
We make the change of variables $u_{l+1}=x_k, u_{l+2}=x_{k-1}+x_k ,\ldots ,u_k=x_{l+1} +\cdots+x_k$ and $v_{n+1}=y_k, v_{n+2}=y_{k-1}+y_k ,\ldots ,v_k=y_{n+1} +\cdots+y_k$.
Then, it follows from Lemma \ref{calculation1} that for $m \in \mathbb{Z}_{\geq 0}$,
\begin{eqnarray*}
I(m+1;k,-n)
&=&  \frac{1}{\Gamma(m+1)} \sum_{l=0}^{n-1} \sum_{j=0}^{n-l-1} A^{(n)}_{l,j} \int_{0}^{\infty} \cdots \int_{0}^{\infty} \sum_{\substack{a_{l+1}+\cdots+a_{k}=m \\ \forall a_i\geq0}} \frac{m!}{a_{l+1}! \cdots a_{k}!} \\
&& \times  x_{l+1}^{a_{l+1}} \cdots x_{k}^{a_k} x_k  \\
&& \times \frac{e^{x_k}}{e^{x_k}-1} \cdots \frac{e^{x_{l+2}\cdots+x_k}}{e^{x_{l+2}+\cdots+x_k}-1} \frac{1}{e^{(n-l-j+1)(x_{l+1} +\cdots+x_k)}} dx_{l+1} \cdots dx_{k} \\
&&  + \frac{1}{\Gamma(m+1)} \int_{0}^{\infty} \cdots \int_{0}^{\infty}  \sum_{\substack{a_{n+1}+\cdots+a_{k}=m \\ \forall a_i\geq0}} \frac{m!}{a_{n+1}! \cdots a_{k}!} y_{n+1}^{a_{n+1}} \cdots y_{k}^{a_k} y_k  \\
&& \times \frac{e^{y_k}}{e^{y_k}-1} \cdots \frac{e^{y_{n+2}\cdots+y_k}}{e^{y_{n+2}+\cdots+y_k}-1} \frac{1}{e^{y_{n+1} +\cdots+y_k}} dy_{n+1} \cdots dy_{k} \\
&=&  \sum_{l=0}^{n-1} \sum_{j=0}^{n-l-1} \sum_{\substack{a_{l+1}+\cdots+a_{k}=m \\ \forall a_i\geq0}} A^{(n)}_{l,j} \frac{1}{\Gamma(a_{l+1}+1) \cdots \Gamma(a_{k}+1)} \\
&& \times \int_{0}^{\infty} \cdots \int_{0}^{\infty} x_{l+1}^{a_{l+1}} \cdots x_{k-1}^{a_{k-1}} x_k^{a_k+1} \\
&& \times \frac{e^{x_k}}{e^{x_k}-1} \cdots \frac{e^{x_{l+2}\cdots+x_k}}{e^{x_{l+2}+\cdots+x_k}-1} \frac{1}{e^{(n-l-j+1)(x_{l+1} +\cdots+x_k)}} dx_{l+1} \cdots dx_{k} \\
&& + \sum_{\substack{a_{n+1}+\cdots+a_{k}=m \\ \forall a_i\geq0}} \frac{1}{\Gamma(a_{n+1}+1) \cdots \Gamma(a_{k}+1)} \\
&& \times \int_{0}^{\infty} \cdots \int_{0}^{\infty}  y_{n+1}^{a_{n+1}} \cdots y_{k-1}^{a_{k-1}} y_k^{a_k+1} \\
&& \times \frac{e^{y_k}}{e^{y_k}-1} \cdots \frac{e^{y_{n+2}\cdots+y_k}}{e^{y_{n+2}+\cdots+y_k}-1} \frac{1}{e^{y_{n+1} +\cdots+y_k}} dy_{n+1} \cdots dy_{k} \\
&=& \sum_{l=0}^{n-1} \sum_{j=0}^{n-l-1} \sum_{\substack{a_{l+1}+\cdots+a_{k}=m \\ \forall a_i\geq0}} A^{(n)}_{l,j}(a_k+1) \frac{1}{(n-l-j+1)^{a_{l+1}+1}} \\
&& \times \zeta^{\star} (a_{l+2}+1, \ldots , a_{k-1}+1, a_{k}+2;\{ n-l-j+1 \}^{k-l-1}) \\
&& +\sum_{\substack{a_{n+1}+\cdots+a_{k}=m \\ \forall a_i\geq0}} (a_k+1) \zeta^{\star} (a_{n+2}+1, \ldots ,a_{k-1}+1, a_{k}+2;\{ 1 \}^{k-n-1}).
\end{eqnarray*}
\qed

\begin{rem}
We denote the right-hand side of $\eta(k,-n;m+1)$ of (\ref{eta(k,-n;m+1)eq}) by $-S_1-S_2$. If $n=0$, we define $S_1=0$. Further if $n=k-1$, we define $S_2=\sum_{a_k=0}^{m} (a_k+1)$.
\end{rem}

\begin{rem}
If $k \leq n$, we can show a certain formula by using same method.
\end{rem}

\begin{ex}
If $k=3$ and $n=0$ in Theorem \ref{eta(k,-n;m+1)},, we have
\begin{align*}
\eta(3,0;m+1) = -\sum_{\substack{a_{1}+a_{2}+a_{3}=m \\ \forall a_i\geq0}} (a_3+1) \zeta^{\star}(a_2+1, a_3+1). 
\end{align*}
Setting $m=0$, we obtain $\eta(3,0;1) = -\zeta^{\star}(1,2)$, which implies $\zeta^{\star}(1,2) = 2\zeta(3)$ by Example \ref{exeta(k,-n;s)}.
\end{ex}

\begin{ex}
If $k=2$ and $n=1$ in Theorem \ref{eta(k,-n;m+1)}, we have
\begin{eqnarray*}
\eta(2,-1;m+1)
&=& -\sum_{a_2=0}^{m} (a_2+1) \frac{1}{2^{m-a_2+1}} \zeta^{\star}(a_2+2;2)-(m+1).
\end{eqnarray*}
Setting $m=0$, we obtain $\eta(2,-1;1) = -\frac{1}{2}\zeta(2)-\frac{1}{2}$.
This result corresponds to Example \ref{exeta(k,-n;s)}.
\end{ex}

\begin{ex}
We have $A_{0,0}^{(1)}=1$. Hence, if $k=3$ and $n=1$ in Theorem \ref{eta(k,-n;m+1)}, we have
\begin{eqnarray*}
\eta(3,-1;m+1)
&=& -\sum_{a_1+a_2+a_3=m} (a_3+1) \frac{1}{2^{a_1+1}} \zeta^{\star}(a_2+1,a_3+2; \{2\}^2) \\
&& -\sum_{a_2+a_3=m} (a_3+1) \zeta^{\star} (a_3+2).
\end{eqnarray*}
Setting $m=0$, we obtain $\zeta(3,-1;1)=-\frac{1}{2} \zeta^{\star}(1,2;\{2\}^2)-\zeta(2)$, which implies $\zeta^{\star}(1,2;\{2\}^2)=-\zeta(2)+2\zeta(3)$ by Example \ref{exeta(k,-n;s)}.

\end{ex}


\subsection{$\eta(-n,k;s)$}

\begin{defi} \label{eta(-n,k;s)}
For $k \in \mathbb{Z}_{\geq1}$ and $n \in \mathbb{Z}_{\geq0}$, define
\begin{equation}
\label{eta(-n,k;s)eq}
\eta(-n,k;s) = \frac{1}{\Gamma(s)} \int_{0}^{\infty} t^{s-1} \frac{{\rm Li}_{-n,k}(1-e^t)}{1-e^t}dt
\end{equation}
for $s \in \mathbb{C}$ with $\Re(s)>0$.
\end{defi}

The integral on the right-hand side converges absolutely in the domain $\Re(s)>0$, as is seen from the following lemma.

\begin{lemma} \label{Li_{-n,k}}
For $k \in \mathbb{Z}_{\geq1}$ and $n \in \mathbb{Z}_{\geq0}$, we have
\begin{equation}
\label{Li_{-n,k}eq}
{\rm Li}_{-n,k}(z) = \sum_{l=0}^{n+1} D_{l}^{(n)} {\rm Li}_{k-l}(z),
\end{equation}
where $D_{l}^{(n)}$ is the real number defined by
\begin{eqnarray}
D_{l}^{(n)} = \left\{ 
\begin{array}{ll}
-1 & (l=n=0) \\
0 & (l=0,n>0) \\
\dbinom{n}{l} (-1)^{n-l} \zeta(-n+l) & (l \neq 0,n+1)\\
\dfrac{1}{n+1} & (l=n+1).
\end{array}
\right.
\end{eqnarray}
\end{lemma}

\begin{proof}
For $n \in \mathbb{Z}_{\geq 0}, l \in \mathbb{Z}$ with $0 \leq l \leq n+1$, we can choose  rational numbers $\{D_{l}^{(n)}\}$ such that 
\begin{equation*}
\sum_{m_1=1}^{m_2-1} m_1^{n}= D_0^{(n)} + D_1^{(n)}m_2 + D_2^{(n)}m_2^{2} + \cdots + D_{n+1}^{(n)}m_2^{n+1}
\end{equation*}
for $m_2 \geq 2$ as follows.
By Faulhaber's formula (see \cite{[CG]}), we have
\begin{eqnarray*}
\sum_{m_1=1}^{m_2-1} m_1^{n}
&=& \frac{1}{n+1} \sum_{j=0}^{n} \binom{n+1}{j} B_j (m_2-1)^{n-j+1} \\
&=& n! \left\{ \sum_{j=0}^{n} \frac{1}{j!(n-j+1)!}(-1)^{n-j+1} B_j \right. \\
&& \left. + \sum_{l=1}^{n+1} \sum_{j=0}^{n-l+1} \frac{1}{j!l!(n-j-l+1)!}(-1)^{n-j-l+1} B_j m_2^l \right\}.
\end{eqnarray*}
When $l=0$, we can see
\begin{eqnarray*}
D_0^{(n)}
&=& n! \sum_{j=0}^{n} \frac{1}{j!(n-j+1)!} (-1)^{n-j+1} B_j \\
&=& n! \sum_{j=0}^{n+1} \frac{1}{j!(n-j+1)!} (-1)^{n-j+1} B_j -n!\frac{B_{n+1}}{(n+1)!} \\
&=& \frac{(-1)^{n+1}B_{n+1}-B_{n+1}}{n+1}.
\end{eqnarray*}
Using $B_{2k+1}=0$, we obtain $D_0^{(n)}=0$ for $n>0$ and $D_0^{(0)}=-1$ for $n=0$.
Further when $l=n+1$, we can derive
\begin{equation*}
D_{n+1}^{(n)}
= n!\sum_{j=0}^{0} \frac{1}{j!(n+1)!(-j)!} (-1)^{-j} B_j
= \frac{1}{n+1}.
\end{equation*}
Further when $l \neq 0, n+1$, we can see that
\begin{eqnarray*}
D_l^{(n)}
&=& n! \sum_{j=0}^{n-l+1} \frac{1}{j!l!(n-j-l+1)!} (-1)^{n-j-l+1} B_j \\
&=& \frac{n!(-1)^{n-l+1}}{l!(n-l+1)!} \sum_{j=0}^{n-l+1} (-1)^j \binom{n-l+1}{j} B_j \\
&=& \frac{n!(-1)^{n-l+1}}{l!(n-l+1)!} B_{n-l+1} \\
&=& \binom{n}{l} (-1)^{n-l} \zeta(-n+l),
\end{eqnarray*}
where $B_j=B^{(1)}_j$.
Using $D_{l}^{(n)}$, we can show
\begin{eqnarray}
\label{LiDeq}
{\rm Li}_{-n,k}(z)
&=& \sum_{1 \leq m_1 <m_2}^{} \frac{m_1^{n}z^{m_2}}{m_2^{k}} \nonumber \\
&=& \sum_{m_2=2}^{\infty} \frac{z^{m_2}}{m_2^{k}} \left( D_0^{(n)} + D_1^{(n)}m_2 + D_2^{(n)}m_2^{2} + \cdots + D_{n+1}^{(n)}m_2^{n+1} \right) \nonumber \\
&=& \sum_{l=0}^{n+1} D_l^{(n)} {\rm Li}_{k-l}(z) -z\sum_{l=0}^{n+1} D_l^{(n)},
\end{eqnarray}
where
\begin{equation*}
\sum_{l=0}^{n+1} D_l^{(n)}
= \sum_{l=1}^{n+1} \frac{n!(-1)^{n-l+1}}{l!(n-l+1)!}B_{n-l+1}
= 0
\end{equation*}
for $n>0$, and
\begin{equation*}
\sum_{l=0}^{n+1} D_{l}^{(n)}
= \sum_{l=0}^{1} D_{l}^{(0)}
= 0
\end{equation*}
for $n=0$.
Hence we obtain the proof.
\end{proof}

\begin{ex} \label{exeta(-n,k;s)}
By Definition \ref{eta(-n,k;s)} and Lemma \ref{Li_{-n,k}}, we can derive
\begin{align*}
&\eta(0,1;s) = -s\zeta(s+1)+1, \\
&\eta(-1,3;s) = -\frac{1}{4}(s^2+s+2)\zeta(s+2)-\frac{1}{2}\zeta(s,2)+\frac{1}{2}s\zeta(1,s+1)+\frac{1}{2}s\zeta(s+1), \\
&\eta(-1,1;s) = \frac{1}{4}\frac{1}{2^{s-1}}-\frac{1}{2}.
\end{align*}
Setting $s=1$,
\begin{align*}
&\eta(0,1;1) = -\zeta(2)+1, \\ 
&\eta(-1,3;1)=-\zeta(3)+\frac{1}{2}\zeta(2), \\
&\eta(-1,1;1) = -\frac{1}{4}.
\end{align*}
\end{ex}

\begin{thm} \label{eta(-n,k;-m)}
The function $\eta(-n,k;s)$ can be analytically continued as an entire function, and the values  at nonpositive integers are given by
\begin{equation}
\eta(-n,k;-m) = B_{m}^{(-n,k)} \ \ \ \ (m \in \mathbb{Z}_{\geq0}).
\end{equation}
\end{thm}

\begin{proof}
Similar to Theorem \ref{eta(k,-n;-m)}, we can obtain the proof.
\end{proof}

\begin{thm} \label{eta(-n,k;m+1)}
For $m \in \mathbb{Z}_{\geq0}$, $k \in \mathbb{Z}_{\geq1}$ and $n \in \mathbb{Z}_{\geq0}$ with $k>n$, we have
\begin{equation}
\label{eta(-n,k;m+1)eq1}
\eta(-n,k;m+1) = \sum_{l=0}^{n+1} D_{l}^{(n)} \sum_{\substack{a_1+\cdots+a_{k-l}=m \\ \forall a_j\geq0}} (a_{k-l}+1) \zeta^{\star}(a_1+1,\ldots,a_{k-l-1}+1,a_{k-l}+2),
\end{equation}
and with $k \leq n$, we have
\begin{eqnarray}
\label{eta(-n,k;m+1)eq2}
\eta(-n,k;m+1) 
&=& \sum_{l=0}^{k-1} D_{l}^{(n)} \sum_{\substack{a_1+\cdots+a_{k-l}=m \\ \forall a_j\geq0}} (a_{k-l}+1) \zeta^{\star}(a_1+1,\ldots,a_{k-l-1}+1,a_{k-l}+2) \nonumber \\
&& + \sum_{l=k}^{n+1} D_{l}^{(n)} B_{-k+l}^{(m+1)}.
\end{eqnarray}
\end{thm}

\begin{proof}
By Lemma \ref{Li_{-n,k}}, we can see that
\begin{eqnarray*}
\eta(-n,k;m+1) 
&=& \frac{1}{\Gamma(m+1)} \int_{0}^{\infty} t^{m} \frac{{\rm Li}_{-n,k}(1-e^t)}{1-e^t}dt \\
&=& \frac{1}{\Gamma(m+1)} \int_{0}^{\infty} t^{m} \frac{\sum_{l=0}^{n+1} D_{l}^{(n)} {\rm Li}_{k-l}(1-e^t)}{1-e^t}dt \\
&=&\sum_{l=0}^{n+1} D_{l}^{(n)} \eta_{k-l}(m+1).
\end{eqnarray*}
In the case $k>n$, using (\ref{etam+1eq}), we obtain
\begin{eqnarray*}
\eta(-n,k;m+1) 
&=& \sum_{l=0}^{n+1} D_{l}^{(n)} \sum_{\substack{a_1+\cdots+a_{k-l}=m \\ \forall a_j\geq0}} (a_{k-l}+1) \zeta^{\star}(a_1+1,\ldots,a_{k-l-1}+1,a_{k-l}+2).
\end{eqnarray*}
On the other hand, in the case $k \leq n$, using (\ref{etam+1eq}), (\ref{etaseq}) and (\ref{BB}), we derive
\begin{eqnarray*}
\eta(-n,k;m+1) 
&=& \sum_{l=0}^{k-1} D_{l}^{(n)} \eta_{k-l}(m+1) + \sum_{l=k}^{n+1} D_{l}^{(n)} \eta_{k-l}(m+1) \\
&=& \sum_{l=0}^{k-1} D_{l}^{(n)} \sum_{\substack{a_1+\cdots+a_{k-l}=m \\ \forall a_j\geq0}} (a_{k-l}+1) \zeta^{\star}(a_1+1,\ldots,a_{k-l-1}+1,a_{k-l}+2) \\
&& + \sum_{l=k}^{n+1} D_{l}^{(n)} \mathfrak{B}_{-k+l}^{(m+1)},
\end{eqnarray*}
where ${\mathfrak{B}}_{-k+l}^{(m+1)} = B_{-k+l}^{(m+1)}$.
\end{proof}

\begin{rem}
In Theorem \ref{eta(-n,k;m+1)}, we denote the right-hand side of (\ref{eta(-n,k;m+1)eq1}) by $\sum_{l=0}^{n+1}S_{l}$. If $n=k-1$, we define $S_{n+1}=S_{k}=D_{k}^{(n)}$.
\end{rem}

\begin{ex}
We have $D_0^{(0)}=-1$ and $D_1^{(0)}=1$.
Hence, if $n=0$ and $k=1$ in Theorem \ref{eta(-n,k;m+1)}, we have
\begin{eqnarray*}
\eta(0,1;m+1)
&=& \sum_{l=0}^{1} D_l^{(0)} \sum_{\substack{a_1+\cdots+a_{1-l}=m \\ \forall a_j\geq0}} (a_{1-l}+1) \zeta^{\star}(a_1+1,\ldots,a_{-l}+1,a_{-l+1}+2) \\
&=& - \sum_{a_1=m}^{} \zeta^{\star}(a_1+2) +1.
\end{eqnarray*}
In particular $\eta(0,1;1)=-\zeta(2)+1$, which corresponds to the result in Example \ref{exeta(-n,k;s)}.
\end{ex}


\subsection{$\eta(1, \ldots, 1,-n;s)$}

In this section, we construct the formula similar to (\ref{xim+1eq}) and (\ref{etam+1eq}).

\begin{defi} \label{eta(1,...,1,-n;s)}
For $r \in \mathbb{Z}_{\geq1}$ and $n \in \mathbb{Z}_{\geq0}$, define
\begin{equation}
\label{eta(1,...,1,-n;s)eq}
\eta(\underbrace{1, \ldots ,1}_{r-1},-n;s) = \frac{1}{\Gamma(s)} \int_{0}^{\infty} \frac{t^{s-1}}{1-e^t} {\rm Li}_{\underbrace{\scriptstyle{1, \ldots ,1}}_{r-1},\scriptstyle{-n}}(1-e^t)dt
\end{equation}
for $s \in \mathbb{C}$ with $\Re(s)>0$.
\end{defi}

The integral on the right-hand side converges absolutely in the domain $\Re(s)>0$, as is seen from the following lemma.

\begin{lemma} \label{Li_{1,...,1,-n}}
For $r \in \mathbb{Z}_{\geq1}$ and $n \in \mathbb{Z}_{\geq0}$ with $r>n+1$, we have
\begin{eqnarray}
\label{Li_{1,...,1,-n}eq}
{\rm Li}_{\underbrace{1, \ldots ,1}_{r-1},-n}(z)
&=& z \left\{ \frac{Q_1^{(n)}(z)}{(1-z)^{n+1}}{\rm Li}_{\underbrace{1, \ldots ,1}_{r-1}}(z) + \frac{Q_2^{(n)}(z)}{(1-z)^{n+1}}{\rm Li}_{\underbrace{1, \ldots ,1}_{r-2}}(z) + \cdots \right. \nonumber \\
&& \left. + \cdots + \frac{Q_n^{(n)}(z)}{(1-z)^{n+1}}{\rm Li}_{\underbrace{1, \ldots ,1}_{r-n}}(z) + \frac{Q_{n+1}^{(n)}(z)}{(1-z)^{n+1}}{\rm Li}_{\underbrace{1, \ldots ,1}_{r-n-1}}(z) \right\},
\end{eqnarray}
where $Q_i^{(n)}(z)$ is the polynomial defined by
\begin{equation}
Q_i^{(n)}(z) = \sum_{k=i}^{n+1} \sum_{l=0}^{n-k+1} (-1)^l \stirii{n+1}{k} \stiri{k}{i} \binom{n-k+1}{l} z^{k+l-1}
\end{equation}
for  $1 \leq i \leq n+1$, and  $\stiri{k}{m}$ is the Stirling number of the first kind and $\stirii{n}{k}$ is the second kind.
\end{lemma}

In order to prove the lemma, we show the following.

\begin{lemma} \label{ddzLi}
For $r$ and $k \in \mathbb{Z}_{\geq1}$ with $r>k+1$, we have
\[
\left( \frac{d}{dz} \right)^k {\rm Li}_{\underbrace{1, \ldots ,1}_{r}}(z) = \sum_{m=1}^{k} \frac{1}{(1-z)^k} \stiri{k}{m} {\rm Li}_{\underbrace{1, \ldots ,1}_{r-m}}(z).
\]
\end{lemma}

\begin{proof}
We prove this lemma by induction.
If $k=1$, then we have
\begin{eqnarray*}
\frac{d}{dz} {\rm Li}_{\underbrace{1, \ldots ,1}_{r}}(z)
&=& \frac{1}{1-z} {\rm Li}_{\underbrace{1, \ldots ,1}_{r-1}}(z).
\end{eqnarray*}
For $k \geq 1$, we assume the case of $k$ holds and consider the case of $k+1$ by using the relational expression $\stiri{n+1}{m}=\stiri{n}{m-1}+n\stiri{n}{m}$.
By induction hypothesis, we have
\begin{eqnarray*}
\left( \frac{d}{dz} \right)^{k+1} {\rm Li}_{\underbrace{1, \ldots ,1}_{r}}(z)
&=& \frac{d}{dz} \left( \sum_{m=1}^{k} \frac{1}{(1-z)^k} \stiri{k}{m}  {\rm Li}_{\underbrace{1, \ldots ,1}_{r-m}}(z) \right) \\
&=& \frac{1}{(1-z)^{k+1}} \left( k \sum_{m=1}^{k} \stiri{k}{m} {\rm Li}_{\underbrace{1, \ldots ,1}_{r-m}}(z) + \sum_{m=2}^{k+1} \stiri{k}{m-1}{\rm Li}_{\underbrace{1, \ldots ,1}_{r-m}}(z) \right) \\
&=&  \frac{1}{(1-z)^{k+1}} \left( k \stiri{k}{1} {\rm Li}_{\underbrace{1, \ldots ,1}_{r-1}}(z) \right. \\
&& \left. + \sum_{m=2}^{k} \left( k \stiri{k}{m} + \stiri{k}{m-1} \right)  {\rm Li}_{\underbrace{1, \ldots ,1}_{r-m}}(z) + \stiri{k}{k}  {\rm Li}_{\underbrace{1, \cdots ,1}_{r-k-1}}(z) \right) \\
&=& \sum_{m=1}^{k+1} \frac{1}{(1-z)^{k+1}} \stiri{k+1}{m} {\rm Li}_{\underbrace{1, \ldots ,1}_{r-m}}(z) .
\end{eqnarray*}
\end{proof}

\noindent
{\it Proof of Lemma \ref{Li_{1,...,1,-n}}.}
When $r>n+1$, we have
\begin{eqnarray*}
{\rm Li}_{\underbrace{1, \ldots ,1}_{r-1},-n}(z)
&=& \left( z \frac{d}{dz} \right)^{n+1} {\rm Li}_{\underbrace{1, \ldots ,1}_{r}}(z) \\
&=& \sum_{k=1}^{n+1} \stirii{n+1}{k} z^k \left( \frac{d}{dz} \right)^k {\rm Li}_{\underbrace{1, \ldots ,1}_{r}}(z).
\end{eqnarray*}
By using Lemma \ref{ddzLi}, we derive
\begin{eqnarray*}
{\rm Li}_{\substack{\underbrace{1, \ldots ,1}_{r-1},-n}}(z)
&=& \sum_{k=1}^{n+1} \stirii{n+1}{k} z^k \frac{1}{(1-z)^k} \sum_{i=1}^{k} \stiri{k}{i} {\rm Li}_{\underbrace{1, \ldots ,1}_{r-i}}(z) \\
&=& \frac{z}{(1-z)^{n+1}} \sum_{i=1}^{n+1} \sum_{k=i}^{n+1} z^{k-1}(1-z)^{n-k+1} \stirii{n+1}{k} \stiri{k}{i} {\rm Li}_{\underbrace{1, \ldots ,1}_{r-i}}(z).
\end{eqnarray*}
Therefore we can see that (\ref{Li_{1,...,1,-n}eq}) holds and $Q_i^{(n)}(z)$ is the polynomial written as
\begin{eqnarray*}
Q_{i}^{(n)}(z)
&=& \sum_{k=i}^{n+1} z^{k-1}(1-z)^{n-k+1} \stirii{n+1}{k} \stiri{k}{i} \\
&=& \sum_{k=i}^{n+1} \sum_{l=0}^{n-k+1} (-1)^{l} \stirii{n+1}{k} \stiri{k}{i} \binom{n-k+1}{l} z^{k+l-1}.
\end{eqnarray*}
\qed

\begin{ex} \label{exeta(1,...,1,-n;s)}
By Definition \ref{eta(1,...,1,-n;s)} and Lemma \ref{Li_{1,...,1,-n}}, we can derive
\begin{align*}
&\eta(1,1,-1;s) = \frac{(s+1)s}{2^{s+3}}-\frac{s}{2^{s+1}}+s.
\end{align*}
Setting $s=1$,
\begin{align*}
&\eta(1,1,-1;1) = \frac{7}{8}.
\end{align*}
\end{ex}

\begin{thm} \label{eta(1,...,1,-n;-m)}
The function $\eta(1,\ldots,1,-n;s)$ can be analytically continued as an entire function, and the values at nonpositive integers are given by
\begin{equation}
\eta(\underbrace{1, \ldots ,1}_{r-1},-n;-m) = B_{m}^{(\overbrace{1, \ldots ,1}^{r-1},-n)} \ \ \ \ (m \in \mathbb{Z}_{\geq0}).
\end{equation}
\end{thm}

\begin{proof}
Similar to Theorem \ref{eta(k,-n;-m)}, we can obtain this theorem.
\end{proof}

\begin{thm} \label{eta(1,...,1,-n;m+1)}
For $m \in \mathbb{Z}_{\geq0}$, $r \in \mathbb{Z}_{\geq1}$ and $n \in \mathbb{Z}_{\geq0}$ with $r>n+1$, we have
\begin{equation}
\label{eta(1,...,1,-n;m+1)eq}
\eta(\underbrace{1, \ldots ,1}_{r-1},-n;m+1) = \sum_{l=1}^{n+1}\sum_{j=0}^{n} \binom{m+r-l}{m} (-1)^{r-l} E_{l,j}^{(n)} \frac{1}{(n-j+1)^{m+r-l}},
\end{equation}
where $E_{l,j}^{(n)}$ is the rational number defined by
\begin{equation}
E_{l,j}^{(n)} = \sum_{M=j+1}^{n+1} \sum_{k=l}^{M} (-1)^{M-k+j} \stirii{n+1}{k} \stiri{k}{l} \binom{n-k+1}{M-k} \binom{M-1}{j}.
\end{equation}
\end{thm}

\begin{proof}
Suppose $r>n+1$.
By setting $Q^{(n)}_{i}(z)=\sum_{j=0}^{i}E^{(n)}_{i,j}(1-z)^j$ in Lemma \ref{Li_{1,...,1,-n}}, we obtain
\[
E_{i,j}^{(n)} =  \sum_{M=j+1}^{n+1} \sum_{k=i}^{M} (-1)^{M-k+j} \stirii{n+1}{k} \stiri{k}{i} \binom{n-k+1}{M-k} \binom{M-1}{j}.
\]
We can transform (\ref{eta(1,...,1,-n;s)eq}).
\begin{eqnarray*}
\eta(\underbrace{1, \ldots ,1}_{r-1},-n;s)
&=& \frac{1}{\Gamma(s)} \int_{0}^{\infty} \frac{t^{s-1}}{1-e^t} {\rm Li}_{\underbrace{1, \ldots ,1}_{r-1},-n}(1-e^t)dt \\
&=& \frac{1}{\Gamma(s)} \sum_{l=1}^{n+1} \sum_{j=0}^{n} E_{l,j}^{(n)} \int_{0}^{\infty} t^{s-1} e^{-t(n-j+1)} {\rm Li}_{\underbrace{1, \ldots,1}_{r-l}}(1-e^t) dt \\
&=& \frac{1}{\Gamma(s)} \sum_{l=1}^{n+1} \sum_{j=0}^{n} (-1)^{r-l} E_{l,j}^{(n)} \frac{1}{(r-l)!} \Gamma(s+r-l) \frac{1}{(n-j+1)^{s+r-l-1}}. \\
\end{eqnarray*}
Setting $s=m+1$, we have
\begin{eqnarray*}
\eta(\underbrace{1, \ldots ,1}_{r-1},-n;m+1)
&=& \frac{1}{m!} \sum_{l=1}^{n+1} \sum_{j=0}^{n} (-1)^{r-l} E_{l,j}^{(n)} \frac{1}{(r-l)!} (m+r-l)! \frac{1}{(n-j+1)^{m+r-l}} \\
&=& \sum_{l=1}^{n+1} \sum_{j=0}^{n} \binom{m+r-l}{m} (-1)^{r-l} E_{l,j}^{(n)} \frac{1}{(n-j+1)^{m+r-l}} .
\end{eqnarray*}
This completes the proof.
\end{proof}

\begin{rem}
If $r \leq n+1$, we can show a certain formula by using same method.
\end{rem}

\begin{ex}
We have $E_{1,0}^{(1)}=1,E_{1,1}^{(1)}=0,E_{2,0}^{(1)}=1,E_{2,1}^{(1)}=-1$.
Hence, if $r=3$ and $n=1$ in Theorem \ref{eta(1,...,1,-n;m+1)}, we have
\begin{eqnarray*}
\eta(1,1,-1;m+1)
&=& \sum_{l=1}^{2} \sum_{j=0}^{1} \binom{m-l+3}{m} (-1)^{-l+3} E_{l,j}^{(1)} \frac{1}{(-j+2)^{m-l+4}} \\
&=& \frac{(m+2)(m+1)}{2} \frac{1}{2^{m+3}} -(m+1)\frac{1}{2^{m+2}} +m+1.
\end{eqnarray*}
In particular $\eta(1,1,-1;1)=\frac{7}{8}$, which corresponds to the result in Example \ref{exeta(1,...,1,-n;s)}.
\end{ex}


\section{$\xi$ function}\label{sec-3}

In this section, we consider $\xi$ function, whose indices consist of positive integers and nonpositive integers.

\begin{rem}
For $k \in \mathbb{Z}_{\geq1}$ and $n \in \mathbb{Z}_{\geq0}$, $\xi(k,-n;s)$ cannot be defined.
In fact, even if $k=1$ and $n=0$, we see that
\begin{equation*}
\xi(1,0;s)
= \frac{1}{\Gamma(s)} \int_{0}^{\infty} t^{s-1} {\rm Li}_{1}(1-e^{-t}) dt
= \frac{1}{\Gamma(s)} \int_{0}^{\infty} t^s dt
\end{equation*}
which is not convergent for any $s\in\mathbb{C}$.
\end{rem}

On the other hand, we can give the following definition by Lemma \ref{Li_{-n,k}}.
\begin{defi} \label{xi(-n,k;s)}
For $k \in \mathbb{Z}_{\geq1}$ and $n \in \mathbb{Z}_{\geq0}$ with $k>n+1$, define
\begin{equation}
\label{xi(-n,k;s)eq}
\xi(-n,k;s) = \frac{1}{\Gamma(s)} \int_{0}^{\infty} t^{s-1} \frac{{\rm Li}_{-n,k}(1-e^{-t})}{e^t-1}dt
\end{equation}
for $s\in \mathbb{C}$ with $\Re(s)>0$.
\end{defi}

\begin{rem}
When  $k \leq n+1$, $\xi(-n,k;s)$ cannot be defined because we see 
\begin{eqnarray}
\label{xixxeq}
\xi(-n,k;s)
&=& \frac{1}{\Gamma(s)} \int_{0}^{\infty} t^{s-1} \frac{{\rm Li}_{-n,k}(1-e^{-t})}{e^t-1}dt \nonumber \\
&=& \frac{1}{\Gamma(s)} \int_{0}^{\infty} t^{s-1} \frac{1}{e^t-1} \sum_{l=0}^{n+1} D_l^{(n)} {\rm Li}_{k-l}(1-e^{-t})dt \nonumber \\
&=&  \sum_{l=0}^{n+1} D_l^{(n)} \xi(k-l;s).
\end{eqnarray}
\end{rem}

\begin{ex} \label{exxi(-n,k;s)}
By Definition \ref{xi(-n,k;s)} and Lemma \ref{Li_{-n,k}}, we can derive
\begin{align*}
&\xi(0,2;s) = -\zeta(s,2)-\zeta(s+2)+s\zeta(1,s+1)+s\zeta(s+1), \\
&\xi(-1,3;s) = -\frac{1}{2}\zeta(s,2)-\frac{1}{2}\zeta(s+2)+\frac{1}{2}s\zeta(1,s+1)+\frac{1}{2}s\zeta(s+1).
\end{align*}
Setting $s=1$,
\begin{align*}
&\xi(0,2;1) = -\zeta(3)+\zeta(2), \\
&\xi(-1,3,1) = -\frac{1}{2}\zeta(3)+\frac{1}{2}\zeta(2).
\end{align*}
\end{ex}

\begin{thm} \label{xi(-n,k;-m)}
When $k>n+1$, the function $\xi(-n,k;s)$ can be analytically continued as an entire function, and the values at nonpositive integers are given
\begin{equation}
\xi(-n,k;-m) = (-1)^m C_m^{(-n,k)} \ \ \ \ (m \in \mathbb{Z}_{\geq0}).
\end{equation}
\end{thm}

\begin{proof}
Similar to Theorem \ref{eta(k,-n;-m)}, we can obtain this theorem.
\end{proof}

\begin{thm} \label{xi(-n,k;m+1)}
For $m \in \mathbb{Z}_{\geq0}$, $k \in \mathbb{Z}_{\geq1}$ and $n \in \mathbb{Z}_{\geq0}$ with $k>n+1$, we have
\begin{equation}
\label{xi(-n,k;m+1)eq}
\xi(-n,k;m+1) = \sum_{l=0}^{n+1} D_{l}^{(n)} \sum_{\substack{a_1+\cdots+a_{k-l}=m \\ \forall a_j\geq0}} (a_{k-l}+1) \zeta(a_1+1,\ldots,a_{k-l-1}+1,a_{k-l}+2).
\end{equation}
\end{thm}

\begin{proof}
Substituting (\ref{xixxeq}) for $s=m+1$, we obtain
\[
\xi(-n,k;m+1)=\sum_{l=0}^{n+1} D_{l}^{(n)} \xi(k-l;m+1).
\]
Hence, using (\ref{xim+1eq}), we can prove this theorem.
\end{proof}

\begin{ex}
We have $D_0^{(0)}=-1$ and $D_1^{(0)}=1$.
Hence, if $n=0$ and $k=2$ in Theorem \ref{xi(-n,k;m+1)}, we have
\begin{eqnarray*}
\xi(0,2;m+1)
&=& -\sum_{a_1+a_2=m}^{} (a_{2}+1) \zeta(a_1+1,a_{2}+2) + \sum_{a_1=m}^{} (a_{1}+1) \zeta(a_1+2).
\end{eqnarray*}
In particular $\xi(0,2;1)=-\zeta(1,2) +\zeta(2)$, we implies $\zeta(1,2)=\zeta(3)$ by Example \ref{exxi(-n,k;s)}.
\end{ex}

\begin{ex}
We have $D_0^{(1)}=0,D_1^{(1)}=-\frac{1}{2},D_2^{(1)}=\frac{1}{2}$.
Hence, if $n=1$ and $k=3$ in Theorem \ref{xi(-n,k;m+1)}, we have
\begin{eqnarray*}
\xi(-1,3;m+1)
&=& -\frac{1}{2}\sum_{a_1+a_2=m}^{} (a_{2}+1) \zeta(a_1+1,a_{2}+2) + \frac{1}{2}\sum_{a_1=m}^{} (a_{1}+1) \zeta(a_1+2).
\end{eqnarray*}
In particular $\xi(-1,3;1)=-\frac{1}{2}\zeta(1,2) + \frac{1}{2}\zeta(2)$, we implies $\zeta(1,2)=\zeta(3)$ by Example \ref{exxi(-n,k;s)}.
\end{ex}


\section{$\widetilde{\xi}$ function}\label{sec-4}

In this section, we consider $\widetilde{\xi}$ function, whose indices consist of positive integers and nonpositive integers.

\begin{rem}
For $k \in \mathbb{Z}_{\geq1}$ and $n \in \mathbb{Z}_{\geq0}$, $\widetilde{\xi}(-n,k;s)$ cannot be defined because by Lemma \ref{Li_{-n,k}} we see
\begin{eqnarray*}
\widetilde{\xi}(-n,k;s)
&=& \frac{1}{\Gamma(s)} \int_{0}^{\infty} t^{s-1} \frac{1}{e^{-t}-1} \sum_{l=0}^{n+1} A_l^{(n)} {\rm Li}_{k-l}(1-e^t) dt \\
&=& D_0^{(n)} \widetilde{\xi} (k;s) + \cdots + D_{n+1}^{(n)} \widetilde{\xi} (k-n-1;s).
\end{eqnarray*}
\end{rem}

On the other hand, we can give the following definition by Lemma \ref{Li_{k,-n}}.

\begin{defi} \label{tildexi(k,-n;s)}
For $k \in \mathbb{Z}_{\geq1}$ and $n \in \mathbb{Z}_{\geq0}$ with $k<n$, define
\begin{equation}
\label{tildexi(k,-n;s)eq}
\widetilde{\xi}(k,-n;s) = \frac{1}{\Gamma(s)} \int_{0}^{\infty} t^{s-1} \frac{{\rm Li}_{k,-n}(1-e^{t})}{e^{-t}-1}dt.
\end{equation}
for $s \in \mathbb{C}$ with $\Re(s)>0$.
\end{defi}

\begin{rem}
When $k \geq n$, $\widetilde{\xi}(k,-n;s)$ cannot be defined.
In fact, even if $k=1$ and $n=0$, we see that
\begin{equation*}
\widetilde{\xi}(1,0;s)
= \frac{1}{\Gamma(s)} \int_{0}^{\infty} t^{s-1} {\rm Li}_{1}(1-e^{t}) dt
= -\frac{1}{\Gamma(s)} \int_{0}^{\infty} t^s dt
\end{equation*}
which is not convergent for any $s \in \mathbb{C}$.
\end{rem}

\begin{ex} \label{extildexi(k,-n;s)}
By Definition \ref{tildexi(k,-n;s)} and Lemma \ref{Li_{k,-n}}, we can derive $\widetilde{\xi}(1,-2;s) = -\frac{s-3}{2^s} + s-3$, and
\begin{align*}
&\widetilde{\xi}(1,-2;1) = -1, \\
&\widetilde{\xi}(2,-3;1) = -1, \\
&\widetilde{\xi}(3,-4;1) = -1.
\end{align*}
\end{ex}

\begin{thm} \label{tildexi(k,-n;-m)}
When $k<n$, the function $\widetilde{\xi}(k,-n;s)$ can be analytically continued as an entire function.
And the values at nonpositive integers are given by
\begin{equation}
\widetilde{\xi}(k,-n;-m) = C_m^{(k,-n)} \ \ \ \ (m \in \mathbb{Z}_{\geq0}).
\end{equation}
\end{thm}

\begin{proof}
Similar to Theorem \ref{eta(k,-n;-m)}, we can obtain the proof.
\end{proof}

\begin{thm} \label{tildexi(k,-n;m+1)}
For $m \in \mathbb{Z}_{\geq0}, k \in \mathbb{Z}_{\geq1}$ and $n \in \mathbb{Z}_{\geq0}$ with $k<n$, we have
\begin{eqnarray}
\label{tildexi(k,-n;m+1)eq}
\widetilde{\xi}(k,-n;m+1)
&=& -\sum_{l=0}^{k-2} \sum_{j=0}^{n-l-1} \sum_{\substack{a_{l+1}+\cdots+a_{k}=m \\ \forall a_j\geq0}} {A'}_{l,j}^{(n)} (a_k+1) \frac{1}{(n-l-j)^{a_{l+1}+1}}  \nonumber \\
&& \times \zeta^{\star}(a_{l+2}+1, \ldots, a_{k-1}+1,a_{k}+2; \{n-l-j\}^{k-l-1}) \nonumber \\
&& - \sum_{j=0}^{n-k} {A'}_{k-1,j}^{(n)} \frac{m+1}{(n-k-j+1)^{m+2}} \nonumber \\
&& - \sum_{j=0}^{n-k-1} {A'}_{k,j}^{(n)} \left( \frac{1}{(n-k-j)^{m+1}}- \frac{1}{(n-k-j+1)^{m+1}} \right),
\end{eqnarray}
where ${A'}_{l,j}^{(n)}$ is a certain integer.
\end{thm}

\begin{proof}
When $k<n$, we have
\begin{eqnarray}
\label{Li_{k,-n}eq2}
{\rm Li}_{k,-n}(z) 
&=& z \left\{ \frac{{P'}^{(n)}_{n-1}(z)}{(1-z)^{n+1}}{\rm Li}_{k}(z) + \frac{{P'}^{(n)}_{n-2}(z)}{(1-z)^{n}}{\rm Li}_{k-1}(z) + \cdots \right. \nonumber \\
&& \left. +\cdots+ \frac{{P'}^{(n)}_{n-k}(z)}{(1-z)^{n-k+2}}{\rm Li}_{1}(z) + \frac{{P'}^{(n)}_{n-k-1}(z)}{(1-z)^{n-k+1}}{\rm Li}_{0}(z) \right\},
\end{eqnarray}
where ${P'}_i^{(n)}(z)$ is a certain polynomial of coefficients of integers.
(When $k \geq n$, we have Lemma \ref{Li_{k,-n}}.)
By setting ${P'}^{(n)}_{i}(z)=\sum_{j=0}^{i}{A'}^{(n)}_{n-i-1,j}(1-z)^j$, we have ${P'}^{(n)}_{i}(1-e^t)=\sum_{j=0}^{i}{A'}^{(n)}_{n-i-1,j}e^{tj}$.
Consider the integral
\begin{eqnarray*}
\widetilde{I}(s;k,-n)
&=& \frac{1}{\Gamma(s)} \sum_{l=0}^{k-2} \sum_{j=0}^{n-l-1} {A'}_{l,j}^{(n)} \int_{0}^{\infty} \int_{0}^{u_k} \cdots \int_{0}^{u_{l+2}} u_k^{s-1} u_{l+1} \\
&& \times \frac{e^{u_{l+1}}}{e^{u_{l+1}}-1} \cdots \frac{e^{u_{k-1}}}{e^{u_{k-1}}-1} \frac{1}{e^{(n-l-j)u_k}} du_{l+1} \cdots du_k \\
&& + \sum_{j=0}^{n-k} {A'}_{k-1,j}^{(n)} \frac{s}{(n-k-j+1)^{s+1}} \\
&& + \sum_{j=0}^{n-k-1} {A'}_{k,j}^{(n)} \left( \frac{1}{(n-k-j)^s}-\frac{1}{(n-k-j+1)^{s}} \right).
\end{eqnarray*}
We can transform this formula as follows.
\begin{eqnarray*}
\widetilde{I}(s;k,-n)
&=& - \frac{1}{\Gamma(s)} \sum_{l=0}^{k-2} \sum_{j=0}^{n-l-1} {A'}_{l,j}^{(n)} \int_{0}^{\infty} u_{k}^{s-1} \int_{0}^{u_k} \cdots \left\{ \int_{0}^{u_{l+2}} \frac{{\rm Li}_{1}(1-e^{u_{l+1}})e^{u_{l+1}}}{e^{u_{l+1}}-1} du_{l+1} \right\} \\
&& \times \frac{e^{u_{l+2}}}{e^{u_{l+2}}-1} \cdots \frac{e^{u_{k-1}}}{e^{u_{k-1}}-1} \frac{1}{e^{(n-l-j)u_k}} du_{l+2} \cdots du_k \\
&& + \frac{1}{\Gamma(s)} \sum_{j=0}^{n-k} {A'}_{k-1,j}^{(n)} \int_{0}^{\infty} u_k^{s} \frac{1}{e^{u_k(n-k-j+1)}} du_k \\
&& + \frac{1}{\Gamma(s)} \sum_{j=0}^{n-k-1} {A'}_{k,j}^{(n)} \int_{0}^{\infty} u_k^{s-1} \frac{e^u_k-1}{e^{u_k(n-k-j+1)}} du_k \\
&=& -\frac{1}{\Gamma(s)} \sum_{l=0}^{k} \int_{0}^{\infty} u_k^{s-1} \frac{\sum_{j=0}^{n-l-1} {A'}_{l,j}^{(n)}e^{tj}}{e^{t(n-l)}} {\rm Li}_{k-l}(1-e^u_{k}) du_k \\
&=& -\frac{1}{\Gamma(s)} \sum_{l=0}^{k} \int_{0}^{\infty} u_k^{s-1} \frac{{P'}_{n-l-1}^{(n)}(1-e^t)}{e^{t(n-l)}} {\rm Li}_{k-l}(1-e^u_{k}) du_k \\
&=& -\widetilde{\xi}(k,-n;s).
\end{eqnarray*}
We make the change of variables $u_{l+1}=x_k, u_{l+2}=x_{k-1}+x_k, \ldots ,u_k=x_{l+1} + \cdots+x_k$.
Then, it follows from Lemma \ref{calculation1} that for $m \in \mathbb{Z}_{\geq 0}$,
\begin{eqnarray*}
\widetilde{I}(m+1;k,-n)
&=& \frac{1}{\Gamma(m+1)} \sum_{l=0}^{k-2} \sum_{j=0}^{n-l-1} {A'}^{(n)}_{l,j} \int_{0}^{\infty} \cdots \int_{0}^{\infty} (x_{l+1} +\cdots+x_k)^{m} x_k \\
&& \times \frac{e^{x_k}}{e^{x_k}-1} \cdots \frac{e^{x_{l+2}\cdots+x_k}}{e^{x_{l+2}+\cdots+x_k}-1} \frac{1}{e^{(n-l-j+1)(x_{l+1} +\cdots+x_k)}} dx_{l+1} \cdots dx_{k} \\
&& + \sum_{j=0}^{n-k} {A'}_{k-1,j}^{(n)} \frac{m+1}{(n-k-j+1)^{m+2}} \\
&& + \sum_{j=0}^{n-k-1} {A'}_{k,j}^{(n)} \left(  \frac{1}{(n-k-j)^{m+1}} -\frac{1}{(n-k-j+1)^{m+1}}\right) \\
&=&  \sum_{l=0}^{k-2} \sum_{j=0}^{n-l-1} \sum_{\substack{a_{l+1}+\cdots+a_{k}=m \\ \forall a_i\geq0}} {A'}^{(n)}_{l,j} \frac{1}{\Gamma(a_{l+1}+1) \cdots \Gamma(a_{k}+1)} \\
&& \times \int_{0}^{\infty} \cdots \int_{0}^{\infty} x_{l+1}^{a_{l+1}} \cdots x_{k-1}^{a_{k-1}} x_k^{a_k+1} \\
&& \times \frac{e^{x_k}}{e^{x_k}-1} \cdots \frac{e^{x_{l+2}\cdots+x_k}}{e^{x_{l+2}+\cdots+x_k}-1} \frac{1}{e^{(n-l-j+1)(x_{l+1} +\cdots+x_k)}} dx_{l+1} \cdots dx_{k} \\
&& + \sum_{j=0}^{n-k} {A'}_{k-1,j}^{(n)} \frac{m+1}{(n-k-j+1)^{m+2}} \\
&& + \sum_{j=0}^{n-k-1} {A'}_{k,j}^{(n)} \left( \frac{1}{(n-k-j)^{m+1}} -\frac{1}{(n-k-j+1)^{m+1}} \right) \\
&=& \sum_{l=0}^{k-2} \sum_{j=0}^{n-l-1} \sum_{\substack{a_{l+1}+\cdots+a_{k}=m \\ \forall a_i\geq0}} {A'}^{(n)}_{l,j}(a_k+1) \frac{1}{(n-l-j+1)^{a_{l+1}+1}} \\
&& \times \zeta^{\star} (a_{l+2}+1, \ldots , a_{k-1}+1, a_{k}+2;\{ n-l-j+1 \}^{k-l-1}) \\
&& + \sum_{j=0}^{n-k} {A'}_{k-1,j}^{(n)} \frac{m+1}{(n-k-j+1)^{m+2}} \\
&& + \sum_{j=0}^{n-k-1} {A'}_{k,j}^{(n)} \left( \frac{1}{(n-k-j)^{m+1}}  -\frac{1}{(n-k-j+1)^{m+1}}\right).
\end{eqnarray*}
\end{proof}

\begin{rem}
Using the Eulerian polynomial $\mathcal{E}_{i}(z)$, we obtain
\begin{eqnarray*}
{P'}_{n-m-1}^{(n)}(z) &=& \binom{n}{m} \mathcal{E}_{n-m}(z) \ \ \ \ (0 \leq m \leq k-1), \\
{P'}_{n-k-1}^{(n)}(z) &=& \sum_{j=k}^{n} \binom{n}{j} \mathcal{E}_{n-j}(z) \mathcal{E}_{j-k}(z).
\end{eqnarray*}
Further ${A'}^{(n)}_{l,j}$ can be explicitly written.
However, this is complicated.
\end{rem}

\begin{ex}
We have ${A'}_{0,0}^{(2)}=2,{A'}_{0,1}^{(2)}=-1$ and ${A'}_{1,0}^{(2)}=3$.
Hence, by Theorem \ref{tildexi(k,-n;m+1)}, if $k=1$ and $n=2$,
\begin{eqnarray*}
\widetilde{\xi}(1,-2;m+1)
&=& -2\frac{m+1}{2^{m+2}} +(m+1) -3\left(1-\frac{1}{2^{m+1}} \right).
\end{eqnarray*}
Setting $m=0$, we obtain $\widetilde{\xi}(1,-2;1)=-1$.
This result corresponds to Example \ref{extildexi(k,-n;s)}.
\end{ex}

\begin{ex}
We have ${A'}_{0,0}^{(4)}=24,{A'}_{0,1}^{(4)}=-36,{A'}_{0,2}^{(4)}=14,{A'}_{0,3}^{(4)}=-1,{A'}_{1,0}^{(4)}=24,{A'}_{1,1}^{(4)}=-24,{A'}_{1,2}^{(4)}=4,{A'}_{2,0}^{(4)}=12,{A'}_{2,1}^{(4)}=-6$ and ${A'}_{3,0}^{(4)}=5$.
Hence, by Theorem \ref{tildexi(k,-n;m+1)}, if $k=3$ and $n=4$,
\begin{eqnarray*}
\widetilde{\xi}(3,-4;m+1)
&=&  -\sum_{a_1+a_2+a_3=m} 24(a_3+1) \frac{1}{4^{a_1+1}} {\zeta}^{\star}(a_2+1,a_3+2;{\{4\}}^{2}) \\
&& +\sum_{a_1+a_2+a_3=m} 36(a_3+1) \frac{1}{3^{a_1+1}} {\zeta}^{\star}(a_2+1,a_3+2;{\{3\}}^{2}) \\
&& -\sum_{a_1+a_2+a_3=m} 14(a_3+1) \frac{1}{2^{a_1+1}} {\zeta}^{\star}(a_2+1,a_2+2;{\{2\}}^{2}) \\
&& +\sum_{a_1+a_2+a_3=m} (a_3+1) {\zeta}^{\star}(a_2+1,a_3+2;{\{1\}}^{2}) \\
&& -\sum_{a_2+a_3=m} 24(a_3+1) \frac{1}{3^{a_2+1}} {\zeta}^{\star}(a_3+2;3) \\
&& +\sum_{a_2+a_3=m} 24(a_3+1) \frac{1}{2^{a_2+1}} {\zeta}^{\star}(a_3+2;2) \\
&& -\sum_{a_2+a_3=m} 4(a_3+1) {\zeta}^{\star}(a_3+2;1) \\
&& -12\frac{m+1}{2^{m+2}} +6(m+1) -5 \left( 1- \frac{1}{2^{m+1}} \right).
\end{eqnarray*}
Setting $m=0$, we obtain $\widetilde{\xi}(3,-4;1)=-6{\zeta}^{\star}(1,2;{\{4\}}^{2})+12{\zeta}^{\star}(1,2;{\{3\}}^{2})-7{\zeta}^{\star}(1,2;{\{2\}}^{2})+6{\zeta}^{\star}(1,2;{\{1\}}^{2})-8{\zeta}^{\star}(2;3) + 12{\zeta}^{\star}(2;2)-4{\zeta}^{\star}(2;1) +\frac{1}{2}$, which implies $-6{\zeta}^{\star}(1,2;{\{4\}}^{2})+12{\zeta}^{\star}(1,2;{\{3\}}^{2})-7{\zeta}^{\star}(1,2;{\{2\}}^{2})+6{\zeta}^{\star}(1,2;{\{1\}}^{2})-8{\zeta}^{\star}(2;3) + 12{\zeta}^{\star}(2;2)-4{\zeta}^{\star}(2;1)=-\frac{3}{2}$ by Example \ref{extildexi(k,-n;s)}.
\end{ex}

\begin{rem}
More generally it seems possible to construct $\eta(k_1 \ldots,k_r;s)$ for $k_1, \ldots,k_r \in \mathbb{Z}$, and $\xi(k_1 \ldots,k_r;s)$ and $\widetilde{\xi}(k_1 \ldots,k_r;s)$ for $k_1, \ldots,k_r \in \mathbb{Z}$ under certain conditions in a similar manner. However, these precedures will be remarkably complicated.
\end{rem}


\section*{Acknowledgments}

The author would like to thank Professor Hirofumi Tsumura for his valuable suggestions, and also appreciate Professor Masanobu Kaneko for his valuable advice and encouragement.


\vspace{40pt}

\noindent
{\large T.Hoshi : } \\
Department of Mathematics and Sciences, Tokyo Metropolitan University, 1-1, Minami-Osawa, Hachioji, Tokyo 192-0397 Japan \\
e-mail : tmk9623546@gmail.com 

\end{document}